\documentclass[11pt,leqno,oneside,letterpaper]{amsart}
\usepackage{amssymb,amsmath,amsfonts,latexsym, amscd, amsthm, graphicx, url, epsfig, mathtools, picture, xcolor}
\usepackage[letterpaper,left=1.75truein, top=1truein]{geometry}
\usepackage{color}
\usepackage[colorlinks,linkcolor=blue,citecolor=blue]{hyperref}
\newtheorem{theorem}{Theorem}[section]
\newtheorem{lemma}[theorem]{Lemma}
\newtheorem{proposition}[theorem]{Proposition}

\newtheorem{corollary}[theorem]{Corollary} 

\theoremstyle{definition}
\newtheorem{remark}[theorem]{Remark} 

\newtheorem{example}[theorem]{Example} 
 
\newtheorem{question}[theorem]{Question}

\def\D{\mathrm{D}}

\begin{document}

\title[Dense orderings]{Dense orderings in the space of left-orderings of a group \footnotetext{2000 Mathematics Subject
Classification. Primary 06F15}}

\author[Adam Clay and Tessa Reimer]{Adam Clay and Tessa Reimer}
\thanks{Adam Clay was partially supported by NSERC grant RGPIN-2014-05465}
\address{Department of Mathematics, 420 Machray Hall, University of Manitoba, Winnipeg, MB, R3T 2N2.}
\email{Adam.Clay@umanitoba.ca, reimer46@myumanitoba.ca}
\urladdr{http://server.math.umanitoba.ca/~claya/}

\begin{abstract}
Every left-invariant ordering of a group is either discrete, meaning there is a least element greater than the identity, or dense.  Corresponding to this dichotomy, the spaces of left, Conradian, and bi-orderings of a group are naturally partitioned into two subsets.  This note investigates the structure of this partition, specifically the set of dense orderings of a group and its closure within the space of orderings.  We show that for bi-orderable groups this closure will always contain the space of Conradian orderings---and often much more.  In particular, the closure of the set of dense orderings of the free group is the entire space of left-orderings.
\end{abstract}
\maketitle
\begin{center}
\today
\end{center}

\section{Introduction}
A group $G$ is \textit{left-orderable} if there is a strict total ordering $<$ of its elements such that $g<h$ implies $fg<fh$ for all $f, g,h \in G$.  Stronger than the notion of left-orderability is  Conradian left-orderability: a left-ordering of a group $G$ is said to be \textit{Conradian} if for every pair of elements $g, h \in G$ with $1<g, h$ there exists $n>0$ such that $1 < g^{-1}hg^n$.  This turns out to be equivalent to requiring that $1 < g^{-1}hg^2$ for all such pairs of elements \cite{Navas10}. Stronger still is the requirement that $G$ admit a left-ordering satisfying $g<h$ implies $gf<hf$ for all $f, g,h \in G$, in which case $<$ is called a \textit{bi-ordering} and $G$ is called bi-orderable.  It is straightforward to see that every bi-ordering is a Conradian left-ordering.  Given a left-ordering $<$ of $G$ (resp. Conradian ordering or bi-ordering), the pair $(G, <)$ will be called a left-ordered group (resp. Conradian ordered or bi-ordered).

Every left-ordering of $G$ can be uniquely identified with its \textit{positive cone} $P = \{ g \in G \mid g>1\}$, which is a subset of $G$ satisfying
\begin{enumerate}
\item $P \sqcup P^{-1} \sqcup \{ 1 \} = G$
\item $P \cdot P \subset P$.
\end{enumerate}
Conversely, every subset of $G$ satisfying the two properties above determines a left-ordering via the prescription $g<h$ if and only if $g^{-1}h \in P$.  A positive cone $P$ is the positive cone of a Conradian left-ordering if, in addition to the two properties above, it satisfies: (3a) If $g, h \in P$ then $g^{-1}hg^2 \in P$.  A positive cone $P$ of a left-ordering is the positive cone of a bi-ordering if it satisfies: (3b) $gPg^{-1} \subset P$ for all $g \in G$.

For a fixed group $G$, if we denote the collections of all positive cones of left-orderings, Conradian orderings and bi-orderings of $G$ by $\mathrm{LO}(G)$, $\mathrm{CO}(G)$ and $\mathrm{BO}(G)$ respectively, then we have $\mathrm{BO}(G) \subset \mathrm{CO}(G) \subset \mathrm{LO}(G)$.  Each of these sets can be topologized so as to become a totally disconnected compact Hausdorff space, as follows.

Let $\mathcal{P}(G)$ denote the power set of $G$ and observe that $\mathrm{LO}(G) \subset \mathcal{P}(G)$.  The power set can be identified with $\{0, 1\}^G$, and thus can be equipped with the product topology.  This makes $\mathcal{P}(G)$ into a totally disconnected Hausdorff space, which is compact by Tychonoff's theorem.  One checks that properties (1) and (2) above define a closed subset of $\mathcal{P}(G)$ (similarly for (3a) and (3b)), so that $\mathrm{LO}(G) \subset \mathcal{P}(G)$ is closed, and hence compact when equipped with the subspace topology.  See the beginning of Section \ref{section:abelian} for a description of a subbasis for the topology on $\mathrm{LO}(G)$.  Similarly, each of $\mathrm{CO}(G)$ and $\mathrm{BO}(G)$ are closed and hence compact.   

We call a left-ordering $<$ of a group $G$ \textit{discrete} if every element in $(G, <)$ has an immediate predecessor and successor, which is equivalent to its positive cone $P = \{ g \in G \mid g>1\}$ having a smallest element.  A left-ordering of a group $G$ which is not discrete is \textit{dense}, in the sense that whenever $g, h \in G$ satisfy $g<h$ there exists $f \in G$ with $g<f<h$.  Equivalently, the positive cone of the ordering does not have a least element.   Throughout this note the set of positive cones of dense left-orderings of the group $G$ will be denoted $\mathrm{D}(G)$.

Thus each of the spaces $\mathrm{LO}(G)$, $\mathrm{CO}(G)$ and $\mathrm{BO}(G)$ admits a decomposition into two subsets: the set of dense orderings and the discrete orderings.  Our work investigates how the nesting $\mathrm{BO}(G) \subset \mathrm{CO}(G) \subset \mathrm{LO}(G)$ behaves with regards to this dichotomy.  We show:

\begin{theorem}
If $G$ is a bi-orderable group which is not isomorphic to the integers then $\mathrm{CO}(G) \subset \overline{\D(G)}$.
\end{theorem}

In fact we show something much stronger, which proves that (in many situations) this containment is proper, see Theorem \ref{main result} and the subsequent examples.  When $G$ is nilpotent, it is known that $\mathrm{LO}(G) = \mathrm{CO}(G)$, and so this yields:

\begin{corollary}
Suppose $G$ is a torsion-free nilpotent group which is not isomorphic to the integers.  Then $\mathrm{LO}(G) = \overline{\D(G)}$.
\end{corollary}

Leveraging the full strength of Theorem \ref{main result} also allows for an analysis if the space of orderings of a free group.  We show that every left-ordering of a free group is an accumulation point of orderings whose Conradian souls\footnote{See the discussion preceding Theorem \ref{main result} for an explanation of the Conradian soul of an ordering.} are nontrivial, noncyclic subgroups.  From this we conclude:

\begin{theorem}
Suppose that $F$ is a free group having $n \geq 2$ generators or countably infinitely many generators.  Then $\mathrm{LO}(F) = \overline{\D(F)}$.
\end{theorem}

Our motivation behind these considerations is as follows.  The spaces $\mathrm{LO}(G)$, $\mathrm{CO}(G)$ and $\mathrm{BO}(G)$ are all totally disconnected, compact Hausdorff spaces---in fact, they are metrizable when $G$ is countable.  Therefore when $G$ is countable, each space is homeomorphic to the Cantor set if and only if it is perfect.  As a result there has been a considerable amount of effort in the literature devoted to identifying isolated points and accumulation points in $\mathrm{LO}(G)$ (E.g. \cite{Ito13, Rivas12}).

This effort can be viewed as an initial step towards a more general problem.  Recall that if $X$ is a topological space, $X'$ denotes the set of all accumulation points of $X$.  Set $X^{(0)} = X$, and for each ordinal number $\alpha$ define $X^{(\alpha+1)} = (X^{(\alpha)})'$ and $X^{(\lambda)} = \bigcap_{\alpha<\lambda} X^{(\alpha)}$ if $\lambda$ is a limit ordinal.  Define the \textit{Cantor-Bendixson rank} of $X$ to be the smallest ordinal $\alpha$ such that $X^{(\alpha+1)} = X^{(\alpha)}$, such an $\alpha$ always exists for cardinality reasons.  These notions were used to great success, for example, in showing that every group admits either finitely many or uncountably many left-orderings \cite{Linnell11}.

Viewed through the lens of Cantor-Bendixson ranks and derived subsets, the question of whether or not $\mathrm{LO}(G)$ admits any isolated points becomes a question of whether or not the Cantor-Bendixson rank of $\mathrm{LO}(G)$ is larger than zero.  For example, the spaces of orderings of the braid groups and of various free products with amalgamation admit isolated points and so have Cantor-Bendixson rank larger than zero \cite{Ito13, DD01}; on the other hand the spaces of orderings of free groups, free products with amalgamation and torsion-free abelian groups admit no isolated points and thus have Cantor-Bendixson rank zero \cite{Rivas12, Sikora04}.  For countable groups there is a well-known upper bound on the Cantor-Bendixson rank: since $\mathrm{LO}(G)$ is Polish its Cantor-Bendixson rank is at most a countable ordinal, by the Cantor-Bendixson theorem.  

These matters are connected to the notions of discrete and dense orderings as follows.  One of the main results of \cite{Clay10} is that under mild hypotheses\footnote{Namely that every rank one abelian subgroup of $G$ be isomorphic to the integers.} on the group $G$, we have $\mathrm{D}(G) \subset \mathrm{D}(G)' $.  From this we conclude $\mathrm{D}(G)  \subset \mathrm{LO}(G)^{(\alpha)}$ for all $\alpha$, and thus $\overline{\mathrm{D}(G)} \subset \mathrm{LO}(G)^{(\alpha)}$ for all $\alpha$.   Our study of $\overline{\mathrm{D}(G)}$, therefore, is an attempt to understand the structure of the sets $\mathrm{LO}(G)^{(\alpha)}$ for large $\alpha$ and ultimately determine the Cantor-Bendixson rank of $\mathrm{LO}(G)$ for $G$ in some nontrivial class of groups.  Specifically, the question motivating our work is:

\begin{question}
Let $G$ be a bi-orderable group.  Can $G$ admit a non-isolated point $P \in \mathrm{LO}(G)$ with $P \notin \overline{\mathrm{D}(G)}$?
\end{question}

If the answer to this question is ``no", it would follow that whenever $G$ is bi-orderable the Cantor-Bendixson rank of $\mathrm{LO}(G)$ must be either $1$ or $0$.

\subsection{Organization}  We organize our arguments as follows.  In Section \ref{section:abelian} we prepare some preliminary results concerning torsion-free abelian groups and the distribution of dense and discrete orderings in their spaces of orderings.  In Section \ref{section:biorderable} we apply these results in the study of bi-orderable groups, and discuss several illustrative examples.  Section \ref{free section} deals with the case of free groups.

\section{Discrete orderings of abelian groups}
\label{section:abelian}
When $A$ is a torsion-free abelian group it is known that $\mathrm{LO}(A)$ has no isolated points unless $A$ is rank one abelian.  When $A$ is not rank one abelian but is finitely generated, the set of dense orderings $\mathrm{D}(A) \subset \mathrm{LO}(A)$ is fairly well understood.

\begin{theorem}\cite[Proposition 4.3]{Clay10}
\label{abelian theorem}
Suppose that $A$ is an abelian group. Then $\overline{\mathrm{D}(A)} = \mathrm{LO}(A)$ if and only if $A$ is not isomorphic to the integers.
\end{theorem}

Question 4.6 of \cite{Clay10} then asks the natural question: What can be said of the set of discrete orderings in $\mathrm{LO}(A)$?  We give a partial answer below by mirroring the proof of \cite[Proposition 4.3]{Clay10}.   We will need this result (specifically Corollary \ref{abelian rank one}) for later.

Our main tool in the proof that follows, which we use repeatedly below and elsewhere in this note, is the procedure of ``changing an ordering on a convex subgroup".  For an ordered group $(G, <)$ a subgroup $C \subset G$ is called \textit{convex relative to $<$} if whenever $f \in G$ and $g, h \in C$ the inequalities $g<f<h$ imply $f \in C$.  In this case, if $P$ is the positive cone of the ordering $<$, then one can check that $P' = P\setminus(P \cap C) \cup Q$ is the positive cone of a left-ordering of $G$ for every $Q \in \mathrm{LO}(C)$.  That is, one can replace the portion of $P$ that lies in $C$ with any other positive cone in $C$.  Note also that an ordering $<$ of a group $G$ is discrete with smallest positive element $g$ if and only if $\langle g \rangle$ is a convex subgroup; a property which we also use often in this paper.

We also recall that if $X$ is either $\mathrm{LO}(G)$, $\mathrm{CO}(G)$ or $\mathrm{BO}(G)$, a subbasis for the topology on $X$ is given by the family of sets $U_g = \{ P \in X \mid g \in P \}$ where $g$ ranges over all nonidentity elements of $G$.

\begin{proposition}
\label{abelian discrete}
Suppose that $k \geq 2$ and that $E \subset \mathrm{LO}(\mathbb{Z}^k)$ is the set of discrete orderings.  Then $\overline{E} =  \mathrm{LO}(\mathbb{Z}^k)$.
\end{proposition}
\begin{proof}
For contradiction, suppose $k>1$ is the smallest $k$ for which the claim fails, and choose a nonempty basic open set $\bigcap_{i=1}^nU_{g_i}$ in $\mathrm{LO}(\mathbb{Z}^k)$, say it contains the positive cone $P$ (here $g_i \in \mathbb{Z}^k$ for $i =1, \ldots, n$). Note that we may assume that none of the $g_i$'s are scalar multiples of one another.  Suppose this basic open set contains no discrete orderings. 

Extend the ordering $<$ defined by $P$ to an ordering of $\mathbb{Q}^k$ by declaring $v_1<v_2$ for $v_1,v_2\in\mathbb{Q}^k$ if $mv_1<mv_2$ whenever $mv_1, mv_2\in\mathbb{Z}^k$. Let $H\subset \mathbb{R}^k$  be the subset of elements $x\in\mathbb{R}^k$ where every Euclidean neighbourhood of $x$ contains both positive and negative elements. One can check that $H$ is a hyperplane which divides $\mathbb{R}^k$ into two components $H_-$ and $H_+$, where $H_-$ contains only negative elements of $\mathbb{Q}^k$ and $H_+$ contains only positive elements of $\mathbb{Q}^k$. Thus the elements of $\{g_1,\ldots, g_n\}$ must lie in either $H_+$ or $H$.  There are three cases to consider.

\textbf{Case 1.} Two or more elements of $\{g_1,\ldots,g_n\}$ lie on $H$. In this case $H\cap\mathbb{Z}^k=\mathbb{Z}^m$ for some $1<m<k$. By assumption, the positive cone $P_H=P\cap (H\cap\mathbb{Z}^k) \subset \mathbb{Z}^m$ is an accumulation point of discrete orderings. Enumerate the $g_i$'s so that $g_i\in P_H$ for $i\leq r$. There exists a positive cone $P'_H\in \cap_{i=1}^rU_{g_i}$ corresonding to a discrete ordering. Note that relative to the ordering defined by $P$, the subgroup $H\cap\mathbb{Z}^k$ is a convex subgroup of $\mathbb{Z}^k$.  Thus $P' = (P \setminus P_H) \cup P'_H$ defines a new positive cone on $\mathbb{Z}^k$, which is the positive cone of a discrete ordering since $P'_H$ defines a discrete ordering on $H\cap\mathbb{Z}^k$.  By construction $P' \in \bigcap_{i=1}^nU_{g_i}$, a contradiction.


\textbf{Case 2.} Exactly one of the $g_i$'s, say $g_1$ lies in $H$. In this case $P$ itself defines a discrete ordering of $\mathbb{Z}^k$ since $P \cap H = \langle g_1 \rangle$ is a convex subgroup of the ordering of $\mathbb{Z}^k$. In particular $\bigcap_{i=1}^nU_{g_i}$ contains a discrete ordering, a contradiction.


\textbf{Case 3.} None of the $g_i$'s are contained in $H$. Suppose $H$ has normal vector $\vec{v}=(v_1,\ldots,v_k).$ Let $\epsilon>0$ and choose $\vec{w}=(w_1,\ldots,w_k)\in\mathbb{Q}^k$ with $\|\vec{v}-\vec{w}\|<\epsilon$. Then choose $(y_1,\ldots,y_k)\in\mathbb{Z}^k$ such that $y_iw_i\in\mathbb{Z}$ for each $i=1,\ldots,k$. Choose $j\in\{1,\ldots,k-1\}$ and let $$m_1=\sum_{i=1}^{j}y_iw_i$$
$$m_2=\sum_{i=j+1}^{k}y_iw_i.	
$$
Then $$\vec{x}=(m_2y_1,\ldots, m_2y_j,-m_1y_{j+1},\ldots,-m_1y_k)\in\mathbb{Z}^k$$ satisfies $\vec{w}\cdot \vec{x}=0$. Thus the hyperplane $H'$ with normal vector $\vec{w}$ satisfies $H'\cap\mathbb{Z}^k = \mathbb{Z}^m$ for some $m> 0$, and since we can choose $\epsilon >0$ as small as we please we may suppose that the $g_i$'s all lie to one side of $H'$. By equipping $\mathbb{Z}^m$ with a discrete ordering, we may lexicographically define a discrete ordering $P'$ on $\mathbb{Z}^k$ with each $g_i$ positive.
\end{proof}

Recall that if $G$ is a group with subgroup $H$, then the isolator of $H$ in $G$ is $$I_G(H)=\{g\in G\mid \exists k\in\mathbb{Z}\ such\ that\ g^k\in H\}.$$ 
A subgroup $H$ of $G$ is called \textit{isolated} in $G$ if $I_G(H) =H$.  In general $I_G(H)$ is a subset of $G$ properly containing $H$ which is not a subgroup unless additional hypotheses are imposed on the group $G$.  For instance if $G$ is abelian (or even nilpotent, see  \cite{Rhemtulla84}) then $I_G(H)$ is a subgroup.  See also Lemma \ref{isolator abelian}.   

\begin{corollary}
\label{abelian rank one}
Suppose that $A$ is a torsion-free abelian group, and for each $P \in \mathrm{LO}(A)$ let $C_P$ denote the smallest nontrivial convex subgroup of the ordering corresponding to $P$.  Then the set
\[ \{ P \in \mathrm{LO}(A) \mid C_P \mbox{ is rank one abelian} \}
\]
is dense in $\mathrm{LO}(A)$.
\end{corollary}
\begin{proof}
Let $P \in \mathrm{LO}(A)$ be given, and suppose $\bigcap_{i=1}^n U_{a_i}$ is a basic open neighbourhood of $P$.  Let $H = \langle a_1, \ldots, a_n \rangle$, and let $Q' = P \cap H$.  By Theorem \ref{abelian theorem} there exists a positive cone $Q \subset H$ with $\{ a_1, \ldots, a_n \} \subset Q$ with $Q \neq Q'$ which corresponds to a discrete ordering of $H$, say with smallest positive element $h \in Q$.

Observe that the positive cone $Q$ extends uniquely to a positive cone $\overline{Q}$ of the subgroup $I_A(H)$, by declaring that $a \in \overline{Q}$ if and only if there exists $k > 0 $ such that $a^k \in Q$.  One can check that if $C \subset H$ is a convex subgroup of the ordering induced by $Q$, then $I_A(C) \subset I_A(H)$ is a convex subgroup of the ordering induced by $\overline{Q}$.  Thus $I_A(\langle h \rangle)$, becomes the smallest nontrivial convex subgroup of $I_A(H)$ relative to the ordering induced by $\overline{Q}$.

Now using the short exact sequence 
\[ 1 \rightarrow I_A(H) \stackrel{i}{\rightarrow} A \stackrel{q}{\rightarrow} A/I_A(H) \rightarrow 1
\]
one can equip $A/I_A(H)$ with an arbitrary positive cone $R$ and set $P' = i(\overline{Q}) \cup q^{-1}(R)$.  By construction $P' \in \bigcap_{i=1}^n U_{a_i} \subset \mathrm{LO}(A)$ and the smallest nontrivial convex subgroup of the corresponding ordering is $I_A(\langle h \rangle)$, which is rank one abelian.
\end{proof}

\section{Dense orderings of bi-orderable groups}
\label{section:biorderable}
In this section we use the property of bi-orderability of $G$ to give sufficient flexibility in the construction of left-orderings of $G$ that we can approximate any Conradian ordering by dense orderings.  In what follows we will use $K$ to denote the Klein bottle group $\langle x, y \mid xyx^{-1} = y^{-1} \rangle$.  

\begin{lemma}
\label{abelian lemma}
Suppose that $G$ is a group that does not contain a copy of the Klein bottle group, and that $P$ is the positive cone of a discrete Conradian ordering of $G$. Suppose that $h$ is the least element of $P$, that $h$ generates the proper normal cyclic subgroup $H \cong \mathbb{Z}$ and that $G/H$ is abelian.  Then $P \in \mathrm{LO}(G)$ is an accumulation point of dense orderings.
\end{lemma}
\begin{proof}
First observe that for every $g \in G$ we have $[g,h]=1$.  To see this, note that since $H=\langle h\rangle$ is normal in $G$, every element $g\in G$ satisfies $ghg^{-1} = h^{\pm1}$. In particular, if $ghg^{-1}=h^{-1}$ one can check that $G$ would contain $K$, which we assume is not possible.  
From this, it follows that if $G/H$ is rank one abelian, then $G$ itself would be abelian as $H$ is cyclic.  Thus $G$ is a torsion-free rank two abelian group, and so the result follows from Theorem \ref{abelian theorem}.

On the other hand, suppose $G/H$ is torsion free abelian of rank larger than two.  Let $g_1, \ldots, g_n$ be finitely many elements of $P$.  We will produce a positive cone $Q$ corresponding to a dense ordering which contains $g_1, \ldots, g_n$.

Since there is a short exact sequence with $H$ convex, we have 
\[ 1 \rightarrow H \stackrel{i}{\rightarrow} G \stackrel{q}{\rightarrow} G/H \rightarrow 1
\]
and $P$ is constructed as $P = \{ h^k \}_{k > 0} \cup q^{-1}(P')$ for some positive cone $P' \subset G/H$.  Suppose that $g_1, \ldots, g_n$ are enumerated so that $g_1, \ldots, g_r$ are powers of $h$ and $g_{r+1}, \ldots, g_n$ lie in $q^{-1}(P')$, meaning $q(g_i) \in P'$ for $r < i \leq n$.

By Proposition \ref{abelian rank one}, we can choose a positive cone $Q'$ of $G/H$ containing $q(g_i)$ for $r < i \leq n$ that produces an ordering with rank one abelian convex subgroup $C \subset G/H$.  The subgroup $q^{-1}(C)$ is abelian of rank two and convex in the ordering whose positive cone is $R=  \{ h^k \}_{k > 0} \cup q^{-1}(Q')$.  The positive cone $R \cap q^{-1}(C)$ contains the elements $g_1, \ldots, g_s$ for some $s \geq r$.  By Theorem \ref{abelian theorem} there is a cone $R' \subset q^{-1}(C)$ containing $g_1, \ldots , g_s$ that is different from $R$, and which defines a dense ordering of $q^{-1}(C)$.  Now set $S = R' \cup q^{-1}(Q' \setminus (C \cap Q'))$, which is the positive cone of a dense ordering of $G$ that contains $g_1, \ldots, g_n$ by construction.
\end{proof}

We need two lemmas concerning isolators of abelian subgroups before moving on to our main theorem.

\begin{lemma}
\label{isolator abelian}
Suppose that $G$ is a bi-orderable group and $A$ is an abelian subgroup.  Then $I_G(A)$ is an abelian subgroup.
\end{lemma}
\begin{proof}
First observe that all elements of $I_G(A)$ commute, because if $[g^n, h^m]=1$ for some $g, h \in G$ then $[g, h]=1$ by bi-orderability.  It then follows that $I_G(A)$ is a subgroup, since $g^k =a \in A$ and $h^{\ell} =b \in A$ implies $(gh)^{k \ell} = a^{\ell}b^k \in A$, and closure under taking inverses is obvious.
\end{proof}

\begin{lemma}\cite[Lemma 3.2]{Clay12}
\label{abelian convex}
Suppose that $A$ is an isolated abelian subgroup of a bi-orderable group $G$.  Then $A$ is relatively convex in $G$.
\end{lemma}

\begin{proposition}
\label{dense biorderable}
Every bi-orderable group which is not isomorphic to the integers admits a dense left-ordering.
\end{proposition}
\begin{proof}
Let $(G,<)$ be a bi-ordered group.  If $<$ is dense, we are done.  Otherwise let $g \in G$ be the least positive element of $<$, and observe that $g$ is central:  Since $1<g$ we know that $1<hgh^{-1}$ for all $h \in G$.  If $h$ does not commute with $g$, this forces $g< hgh^{-1}$ since $g$ is the least positive element.  But then conjugation yields $h^{-1}gh <g$, a contradiction.

Thus $g$ is central, and since $G$ is not infinite cyclic there exists $h \in G$ that is not a power of $g$.   Then $\langle g, h \rangle$ is a rank two abelian subgroup of $G$, by Lemma \ref{isolator abelian} $I_G(\langle g, h \rangle)$ is an isolated abelian subgroup; one can check it also has rank two.   By Lemma \ref{abelian convex} $I_G(\langle g, h \rangle)$ is relatively convex.  Every rank two abelian group admits a dense ordering, so we are done.
\end{proof}

Note that bi-orderability is essential in the previous proposition.  The finitely generated Tararin groups
\[ T_n = \langle x_1, \ldots, x_n \mid x_ix_{i-1}x_i^{-1} = x_{i-1}^{-1} \mbox{ for } i=2,\ldots n \rangle
\]
satisfy $|\mathrm{LO}(T_n)|=2^n$ and all the orderings are discrete.  It is also possible to construct groups having uncountably many orderings, all of them discrete, such as the so-called infinite Tararin group $\langle x_i, i \in \mathbb{N} \mid x_ix_{i-1}x_i^{-1} = x_{i-1}^{-1} \mbox{ for } i \in \mathbb{N}_{>1} \rangle$.  None of these groups are bi-orderable, as each contains an element that is conjugate to its own inverse.

\begin{proposition}
\label{finitely generated case}
Suppose that $G$ is a bi-orderable group that is not isomorphic to the integers and that $P$ is the positive cone of a Conradian ordering of $G$.  Then $P \in \overline{\D(G)}$.
\end{proposition}
\begin{proof}
Suppose that $P$ is the positive cone of a discrete Conradian ordering, and that $P \in \cap_{i=1}^nU_{g_i}$.  Let $h>1$ denote the least element of $P$.

First suppose that there exists a convex subgroup $C$ such that $(\langle h \rangle, C)$ is a convex jump.  Assume that $g_1, \ldots, g_n$ are enumerated so that $g_1, \ldots, g_r \in C$ and $g_{r+1}, \ldots, g_n \notin C$.  By Lemma \ref{abelian lemma}, there exists a positive cone $Q \in \mathrm{LO}(C)$ such that $Q \neq P \cap C$ and $g_1, \ldots, g_r \in Q$.  Then $P' = (P \setminus (P \cap C)) \cup Q$ contains $g_1,\ldots,g_n$, is different from $P$ and is the positive cone of a dense ordering.

On the other hand, suppose that there is no convex subgroup $C$ such that $(\langle h \rangle, C)$ is a convex jump.  Suppose further that $g_1, \ldots, g_n$ are enumerated so that $g_1, \ldots, g_r$ are powers of $h$ and $g_{r+1}, \ldots, g_n$ are not; suppose also that $g_{r+1}$ is the smallest element which is not in $\langle h \rangle$.   Then $g_{r+1}$ determines a convex jump $(C, D)$, note that $g_j \notin C$ for all $j > r$ and that the containment $\langle h \rangle \subset C$ is proper.  To complete the proof, it suffices to observe that $C$ can be equipped with a dense ordering by Proposition \ref{dense biorderable}.  Thus we can choose a positive cone $Q \subset C$ with $h \in Q$ and set $P' = (P \setminus (P \cap C)) \cup Q$ as before.
\end{proof}

\begin{corollary}
If $G$ is a torsion-free nilpotent group that is not isomorphic to the integers, then $\mathrm{LO}(G) = \overline{\mathrm{D}(G)}$.
\end{corollary}
\begin{proof}
Every torsion-free nilpotent group is bi-orderable, and all left-orderings of every torsion-free nilpotent group are Conradian \cite{Ault72}.
\end{proof}

We can extend the previous proposition so that it applies to certain orderings of non-bi-orderable groups.  Indeed, as Example \ref{example braids} shows, the group $G$ need not even be locally indicable for our generalized result to apply.  For the statement of our theorem below, recall that the \textit{Conradian soul} of an ordering $<$ of a group $G$ is the largest convex subgroup $C \subset G$ such that the restriction of $<$ to $C$ is Conradian.

\begin{theorem}
\label{main result}
Suppose that $P$ is the positive cone of a left-ordering of a bi-orderable group $G$. If the Conradian soul of the ordering corresponding to $P$ is bi-orderable, nontrivial, and not isomorphic to $\mathbb{Z}$, then $P \in \overline{\mathrm{D}(G)}$.
\end{theorem}
\begin{proof}
With $P$ and $G$ as in the statement of the theorem, suppose that $P \in \bigcap_{i=1}^nU_{g_i}$ where $g_1, \ldots, g_n \in G$. Let $C$ denote the Conradian soul of the ordering corresponding to $P$, and suppose that the $g_i$'s are enumerated so that $g_1, \ldots, g_r \in C$ and $g_i \notin C$ for $i= r+1, \ldots, n$.  By Proposition \ref{finitely generated case} there is a positive cone $Q \subset C$ containing $g_1, \ldots g_r$ whose corresponding ordering of $C$ is dense.  Then $P' = P \setminus(P \cap C) \cup Q$ is the positive cone of a dense ordering of $G$, and $P' \in  \bigcap_{i=1}^nU_{g_i}$ by construction.
\end{proof}

With this generalization it is straightforward to construct orderings of non-biorderable groups that are accumulation points of dense orderings.

\begin{example}
\label{example braids}
Recall that $$B_n = \langle \sigma_1, \ldots, \sigma_{n-1} \mid \sigma_i \sigma_j \sigma_i = \sigma_j \sigma_i \sigma_j \mbox{ if $|i-j| =1$ and } \sigma_i \sigma_j = \sigma_j \sigma_i \mbox{ otherwise} \rangle.$$
The positive cone $P_D$ of the Dehornoy ordering $<_D$ of $B_n$ is defined as follows.   Given a word $w$ in the generators $\sigma_i$, we say that $w$ is $i$-positive if it contains no occurrences of $\sigma_j$ for $j <i$ and all occurrences of $\sigma_i$ (of which there must be at least one) occur with positive exponent.  A braid $\beta \in B_n$ lies in $P_D$ if and only if it admits a representative word $w$ that is $i$-positive for some $i$.

Fix $n >4$ and consider $B_n$.  The convex subgroups of $<_{D}$ are precisely the subgroups $\langle \sigma_r, \ldots, \sigma_{n-1}\rangle$ with $r \geq 1$ \cite{DDRW08}.  In particular, $\langle \sigma_{n-2}, \sigma_{n-1} \rangle \cong B_3$ is a proper convex subgroup.  Equip this copy of $B_3$ with any left-ordering whose Conradian soul is contained in $[B_3, B_3] \cong F_2$ and is not infinite cyclic.  Extend this ordering to $B_n$ using the Dehornoy ordering outside of $\langle \sigma_{n-2}, \sigma_{n-1} \rangle$.  By Theorem \ref{main result} the resulting ordering is an accumulation point of dense orderings of $B_n$.  However, $B_n$ itself is not bi-orderable---in fact, not even Conradian left-orderable since $[B_n, B_n]$ is finitely generated and perfect for $n \geq 5$. \qed
\end{example}

\begin{remark}
A family of left-orderings of $B_n$ of particular interest are the \textit{Nielsen-Thurston orderings}.  These are the orderings which arise from considering the action of $B_n$, thought of as a mapping class group, on the boundary of the universal cover of the $n$-punctured disk equipped with a hyperbolic metric (see \cite[Chapter XIII]{DDRW08} for more details).  Such orderings are either of finite or infinite type, depending on how a certain geodesic which describes the ordering cuts up the $n$-punctured disk.  The authors of \cite{NW11} show that the Nielsen-Thurston orderings of infinite type are dense, while those of finite type have Conradian soul isomorphic to $\mathbb{Z}^k$ for $k \geq 1$.  When $k >1$ such orderings are obviously an accumulation point of dense orderings, but when $k=1$ the picture is not so clear (though it is known that these orderings are not isolated points).  It may be of some interest to determine whether or not the Nielsen-Thurston orderings with Conradian soul isomorphic to $\mathbb{Z}$ lie in $\overline{\mathrm{D}(B_n)}$, as this would imply that \textit{all} Nielsen-Thurston orderings lie in $\overline{\mathrm{D}(B_n)}$.
\end{remark}

\section{Free groups}
\label{free section}

Let $F_n$ denote the free group on generators $\{ x_1, \dots, x_n \}$.  In this section we show that $\mathrm{LO}(F_n) = \overline{\mathrm{D}(F_n)}$, which will follow as a corollary of the following theorem:

\begin{theorem}
\label{soul theorem}
Let $n \geq 2$ and suppose that $P \in \bigcap_{i=1}^m U_{g_i} \subset \mathrm{LO}(F_n)$ for some collection of nonidentity elements $g_1, \ldots, g_m \in F_n$.  Then there exists $Q \in \bigcap_{i=1}^m U_{g_i}$ and a subgroup $C \subset F_n$ with $g_i \notin F_n$ for all $i$, satisfying:
\begin{enumerate}
\item $C$ is convex relative to the ordering of $F_n$ determined by $Q$;
\item $C$ is nontrivial and not isomorphic to $\mathbb{Z}$.
\end{enumerate}
\end{theorem}

Some of the details of the proof are a special case of computations done in \cite{Rivas12}, and so are omitted here for clarity of exposition.

\begin{proof}
Corresponding to the positive cone $P$ there is a left-ordering $<$ of $F_n$ with $1<g_i$ for $i=1, \ldots, m$.  Let $\rho : F_n \rightarrow \mathrm{Homeo}_+(\mathbb{R})$ denote the dynamic realization of $<$, which is a representation satisfying $1<g$ if and only if $\rho(g)(0)>0$ for all $g \in F_n$.  For the rest of this proof, let $B_k$ denote the $k$-ball in $F_n$ relative to the generating set $\{x_1, \ldots, x_n\}$.

We will show how to construct, for each $k \geq 1$, a representation $\rho_k : F_n \rightarrow \mathrm{Homeo}_+(\mathbb{R})$ satisfying:
\begin{enumerate}
\item $\rho_k(w)(0) = \rho(w)(0)$ for all $w \in B_k$, and
\item there exist nonidentity elements $h_1, h_2 \in F_n$ such that $\langle h_1, h_2 \rangle$ is not cyclic and $\rho_k(h_i)(0) = 0$ for $i=1,2$.
\end{enumerate}
Having constructed such representations, the theorem follows:  First choose an enumeration of the rationals $\{r_0, r_1, r_2, \dots \}$ with $r_0 =0$ and define a left-ordering $\prec$ of  $\mathrm{Homeo}_+(\mathbb{R})$ according to the rule $f_1\prec f_2$ if and only if $f_1(r_i) < f_2(r_i)$, where $r_i$ is the first rational in the enumeration $\{r_0, r_1, r_2, \dots \}$ with $f_1(r_i) \neq f_2(r_i)$.  The stabilizer of $0$, $Stab(0)$, is a convex subgroup in this left-ordering.  Now choose $k \geq 1$ such that $g_1, \ldots, g_m \in B_k$.  Then with the representation $\rho_k$ constructed as above, consider the short exact sequence $1 \rightarrow \ker(\rho_k) \rightarrow F_n \rightarrow \rho_k(F_n) \rightarrow 1$, and lexicographically order $F_n$ using the restriction of $\prec$ to $\rho_k(F_n)$ and whatever ordering one pleases on $\ker(\rho_k)$.  Then $C = \langle h_1, h_2 \rangle$ is not cyclic and is contained in $\rho_k^{-1}(Stab(0))$, which is a convex subgroup relative to the resulting ordering of $F_n$.  Moreover if we use $Q$ to denote the positive cone of this ordering of $F_n$, then by our choice of $\rho_k$ we have $\rho_k(g_i)(0) = \rho(g_i)(0) >0$ for all $i =1, \ldots, m$ and hence $Q \in \bigcap_{i=1}^m U_{g_i}$.

Thus we fix $k \geq 1$ and focus on constructing $\rho_k$ as above.  Let $g^+ = \max B_k $ and $g^- = \min B_k$, where the maximum and minimum are taken relative to the ordering $<$ of $F_n$ restricted to $B_k$.   Since the dynamic realization $\rho$ satisfies $\rho(h)(0) >0$ if and only if $h>1$, the assignment $h \mapsto \rho(h)(0)$ is order-preserving.  We conclude that $\rho(w)(0) \in [\rho(g^-)(0), \rho(g^+)(0)]$ for all $w \in B_k$.  From this it follows, by induction on the length of $w$, that if $\rho_k$ satisfies $\rho_k(x_i)(y) = \rho(x_i)(y)$ for $i = 1, \ldots, n$ and for all $y \in  [\rho(g^-)(0), \rho(g^+)(0)]$, then $\rho_k(w)(0) = \rho(w)(0)$ for all $w \in B_k$ (this is a special case of Lemma 1.9 in \cite{Rivas12}).

Now for each $j = 1, \ldots, m$ choose $\epsilon_j = \pm 1$ such that $x_j^{\epsilon_j} g^+ >g^+$, and choose $j_0$ such that $x_{j_0}^{\epsilon_{j_0}} g^+ = \min \{x_1^{\epsilon_1} g^+, \ldots , x_n^{\epsilon_n} g^+\}$.  To simplify notation, set $a = x_{j_0}^{\epsilon_{j_0}}$.  Since $n \geq 2$ we may choose $\ell \neq j_0$, and set $b = x_{\ell}^{\epsilon_{\ell}}$. 

For ease of notation in the arguments below, in place of $\rho(h)(x)$ we simply write $h(x)$ whenever $h \in F_2$ and $x \in \mathbb{R}$.  Define order-preserving homeomorphisms $f_1, f_2 : \mathbb{R} \rightarrow \mathbb{R}$ as follows:
\[
    f_1(x)= 
\begin{cases}
   a(x) \text{ if } x \leq g^+(0);&\\
    \left( \cfrac{bg^+(0) - ag^+(0)}{ag^+(0) - g^+(0)} \right)(x-ag^+(0))+bg^+(0)              & \text{otherwise}.
\end{cases}
\]
Then noting that $f_1 (bg^+(0)) > bg^+(0)$, set:
\[
    f_2(x)= 
\begin{cases}
    b(x)  \text{ if } x \leq g^+(0);&\\
    \left( \cfrac{f_1(bg^+(0)) - bg^+(0)}{bg^+(0) - g^+(0)} \right)(x-bg^+(0))+f_1(bg^+(0))              & \text{otherwise}.
\end{cases}
\]
See Figures \ref{figf1} and \ref{figf2} for graphical explanations of these functions; note that $g^+(0) < ag^+(0) < bg^+(0)$ and $g^+(0) < bg^+(0) < f_1(bg^+(0))$ follow from our choices of $j_0$, $\epsilon_{j_0}$, and $\epsilon_{\ell}$.

\begin{figure}[h!]
\begin{center}
\setlength{\unitlength}{1cm}
\thicklines
\begin{picture}(12,5)
\put(1,0){\colorbox{gray!20}{\makebox(4.8,2.9){$a(x)$ here}}}
\put(5,0){\line(0,1){5}}
\put(1,1){\line(1,0){8}}
\multiput(6,0)(0, 0.125){24}{\line(0,1){0.05}}
\multiput(1,3)(0.125, 0){40}{\line(1,0){0.05}}
\put(6, 0.9){\line(0,1){0.2}}
\put(8, 0.9){\line(0,1){0.2}}
\put(4.9, 3){\line(1,0){0.2}}
\put(4.9, 4){\line(1,0){0.2}}
\put(6, 3){\line(2,1){3}}
\put(6, 3){\circle*{0.1}}
\put(6.25, 2.75){$(g^+(0), ag^+(0))$}
\put(8, 4){\circle*{0.1}}
\put(8.25, 3.75){$(ag^+(0), bg^+(0))$}
\end{picture}
\caption{The function $f_1(x)$.}
\label{figf1}
\end{center}
\end{figure}
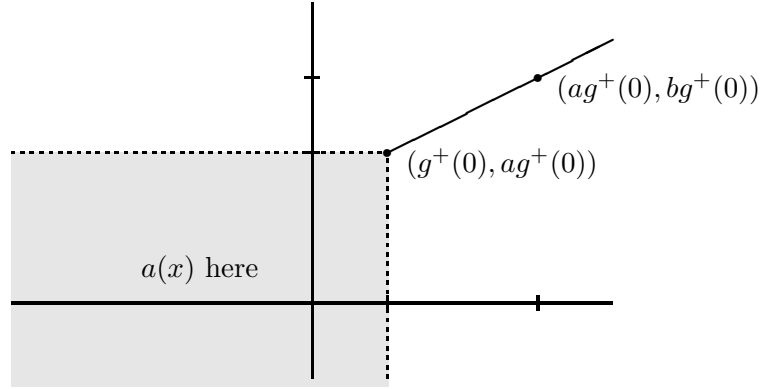

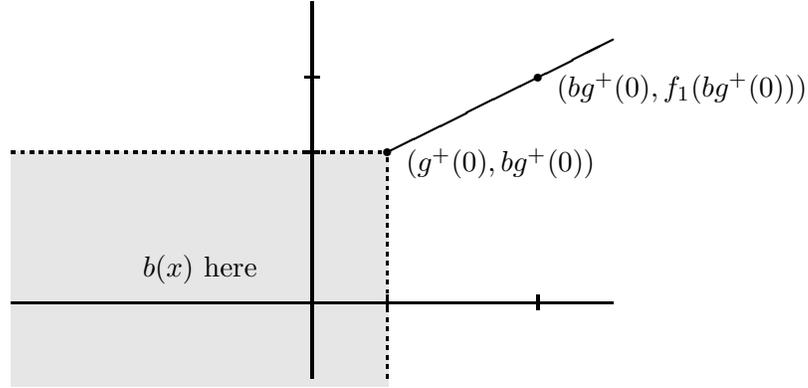
\begin{figure}[h!]
\setlength{\unitlength}{1cm}
\thicklines
\begin{picture}(12,5)
\put(1,0){\colorbox{gray!20}{\makebox(4.8,2.9){$b(x)$ here}}}
\put(5,0){\line(0,1){5}}
\put(1,1){\line(1,0){8}}
\multiput(6,0)(0, 0.125){24}{\line(0,1){0.05}}
\multiput(1,3)(0.125, 0){40}{\line(1,0){0.05}}
\put(6, 0.9){\line(0,1){0.2}}
\put(8, 0.9){\line(0,1){0.2}}
\put(4.9, 3){\line(1,0){0.2}}
\put(4.9, 4){\line(1,0){0.2}}
\put(6, 3){\line(2,1){3}}
\put(6, 3){\circle*{0.1}}
\put(6.25, 2.75){$(g^+(0), bg^+(0))$}
\put(8, 4){\circle*{0.1}}
\put(8.25, 3.75){$(bg^+(0), f_1(bg^+(0)))$}
\end{picture}
\caption{The function $f_2(x)$.}
\label{figf2}
\end{figure}

Define $\rho_k : F_n \rightarrow \mathrm{Homeo}_+(\mathbb{R})$ as follows.  For $i \notin \{ j_0, \ell \}$ set  $\rho_k(x_i) = x_i$, and set $\rho_k(a) = f_1$ and $\rho_k(b) = f_2$.  Next set $h_1 = (bg^+)^{-1}a^2g^+$ and $h_2 = (abg^+)^{-1}b^2g^+$.  Observe that $bg^+$ and $ag^+$ are reduced words, since the exponents $\epsilon_{j_0}$ and $\epsilon_{\ell}$ are chosen so that $bg^+(0), ag^+(0) \notin [g^-(0), g^+(0)]$.  Since $j_0 \neq \ell$ it follows that $h_1$ and $h_2$ are reduced words in the generators $\{ x_1, \ldots, x_n \}$, and we conclude that $h_1, h_2$ do not represent the identity.  Moreover there are no integers $s, t$ such that $h_1^s = h_2^t$, because the commutator $[h_1, h_2]$ is not the identity, so $\langle h_1, h_2 \rangle$ is not cyclic.

Lastly, by using the facts: (1) $\rho_k(g^+)^{\pm1}(0) = (g^+)^{\pm1}(0)$ and (2) $\rho_k(a)(x) = f_1(x)$ and $\rho_k(b)(x) = f_2(x)$ for all $x \in \mathbb{R}$, one computes that $\rho_k(h_1)(0) = 0$ and $\rho_k(h_2)(0)=0$.  This completes the proof.
\end{proof}

\begin{corollary}
\label{free corollary}
If $n \geq 2$ then $\mathrm{LO}(F_n) = \overline{\mathrm{D}(F_n)}$.
\end{corollary}
\begin{proof}
Suppose that $P$ is the positive cone of a left-ordering of $F_n$, and that  $P \in \bigcap_{i=1}^m U_{g_i}$ for some $g_1, \ldots, g_m \in F_n$.  Choose $Q \in \bigcap_{i=1}^m U_{g_i}$ with corresponding subgroup $C \subset F_n$ as in the conclusion of Theorem \ref{soul theorem}.   Choose a bi-ordering of $C$ with positive cone $R$, and set $Q' = Q \setminus (Q \cap C) \cup R$.  Then $Q'\in \bigcap_{i=1}^m U_{g_i}$, and $Q'$ corresponds to a left-ordering of $F_n$ whose Conradian soul contains $C$.  In particular, Theorem \ref{main result} implies that $Q' \in   \overline{\mathrm{D}(F_n)}$.  

It follows that the positive cone $P$ is an accumulation point of elements of $ \overline{\mathrm{D}(F_n)}$, so $P \in  \overline{\mathrm{D}(F_n)}$.
\end{proof}

While the previous proof can be modified to handle the case of $F_{\infty}$ (the free group with countably infinitely many generators), the space $\mathrm{LO}(F_{\infty})$ can also be analyzed directly as below.

\begin{example}
Let $F_{\infty}$ denote the free group on countably many generators $\{x_i\}_{i \in \mathbb{N}}$.  Then $\mathrm{LO}(F_{\infty})$ is homeomorphic to the Cantor set:  If $P \in \bigcap_{i=1}^n U_{g_i} \subset \mathrm{LO}(F_{\infty})$, choose $k$ large enough that $x_i$ for $i \geq k$ does not occur in any reduced word representing $g_1, \ldots, g_n$.  Then the automorphism $\phi:F_{\infty} \rightarrow F_{\infty}$ defined by $\phi(x_i)=x_i$ for $i \neq k$ and $\phi(x_k) = x_k^{-1}$ yields a positive cone $\phi(P) \neq P$ that contains $g_1, \ldots, g_n$.

We can in fact approximate such a positive cone $P$ by dense orderings of $\mathrm{LO}(F_{\infty})$.  With $k$ as above, consider the map $h: F_{\infty} \rightarrow \langle x_1, \ldots x_{k-1} \rangle \cong F_{k-1}$ given by $h(x_i) = x_i$ for $i < k$ and $h(x_i) = 1$ for $i>k$. Equip $F_{k-1}$ with the positive cone $P \cap F_{k-1}$, and the subgroup $\langle x_k, x_{k+1}, \ldots \rangle$ with any positive cone $Q$ corresponding to a dense ordering of $\langle x_k, x_{k+1}, \ldots \rangle \cong F_{\infty}$.  Now using the short exact sequence 
\[ 1 \rightarrow \langle x_k, x_{k+1}, \ldots \rangle \stackrel{i}{\rightarrow} F_{\infty} \stackrel{h}{\rightarrow} F_{k-1} \rightarrow 1
\]
we lexicographically order $F_{\infty}$ using the positive cone $P' = i(Q) \cup h^{-1}(P \cap F_{k-1})$.  The result is a positive cone $P' \in \bigcap_{i=1}^n U_{g_i}$ whose corresponding ordering is dense.  We conclude that $\overline{\mathrm{D}(F_{\infty})} = \mathrm{LO}(F_{\infty})$. \qed
\end{example}

\bibliographystyle{plain}

\bibliography{dense}

\end{document}